\documentclass{article}
\usepackage{fullpage}
\usepackage{graphics}
 \usepackage{graphicx,xcolor}	
\usepackage{amssymb}
 \usepackage{amsmath}
 \usepackage{amsthm,lmodern}

 \newtheorem{theorem}{Theorem}[section]
\newtheorem{lemma}{Lemma}[section]
\newtheorem{proposition}[lemma]{Proposition}

\renewcommand{\H}{\mathcal{H}} 

\newcommand{\Lra}{\Longrightarrow}

\newcommand{\sh}{\sharp}
\newcommand{\pd}{\partial}

\newcommand{\fal}{\forall}

\newcommand{\vep}{\varepsilon} 
 
\newcommand{\dt}{\delta} 
\newcommand{\al}{\alpha}

\newcommand{\ggm}{\Gamma}

\newcommand{\om}{\Omega}
\newcommand{\la}{\lambda} 
\renewcommand{\d}{\,\text{d}}
\newcommand{\TT}{{\mathbb{T}^2}}

 \title{On a Quaternary Non-Local Isoperimetric Problem}

\author{Stanley Alama
\thanks{
Department of Mathematics and Statistics, McMaster University. E-mail: alama@mcmaster.ca} 
 \qquad Lia Bronsard 
 \thanks{
Department of Mathematics and Statistics, McMaster University. E-mail: bronsard@mcmaster.ca } 
 \qquad Xinyang Lu
 \thanks{
 Department of Mathematical Sciences, Lakehead University. 
 Email: xlu8@lakeheadu.ca
 }
 \qquad Chong Wang
 \thanks{
 Department of Mathematics, Washington and Lee University. Email: cwang@wlu.edu
 }
}

 \date{}							
 
\begin{document}
 
 \maketitle

 This paper is a heartfelt tribute to Bob Pego, acknowledging his profound and enduring contributions to the field of PDEs and applied mathematics.
 \\

\begin{abstract}

We study a two-dimensional quaternary inhibitory system. This free energy functional combines
an interface energy favoring micro-domain growth with a 
Coulomb-type long range interaction energy which prevents micro-domains from unlimited spreading.  Here we consider a limit in which three species are vanishingly small, but interactions are correspondingly large to maintain a nontrivial limit. In this limit two energy levels are distinguished:  the highest order limit encodes information on the geometry of local structures as a three-component isoperimetric problem,
while the second level describes the spatial distribution of components in global minimizers.  Geometrical descriptions of limit configurations are derived. 
	
\end{abstract}

 \numberwithin{equation}{section}             


\section{Introduction}

Block copolymers, a class of soft materials, have gained much interest in recent years due to their ability of self-organization into nanoscale ordered structures \cite{bf}.
An $ABCD$ tetrablock copolymer is a linear-chain molecule consisting of four subchains, joined covalently to each other.
Because of the repulsive forces between different types of monomers, different types of subchain tend to segregate. However, since subchains are chemically bonded in molecules, segregation can lead to a phase separation only at microscopic level, 
where $A, B, C$ and $D$-rich micro-domains emerge, forming morphological phases, many of which have been observed experimentally.


To model tetrablock copolymers, one uses a free energy
\begin{eqnarray}  \label{energyu}
\mathcal{E} (u) :=  \frac{1}{2} \sum_{i=0}^3 \int_{\mathbb{T}^2} |\nabla u_i | +  \sum_{i,j = 1}^3  \frac{\gamma_{ij}}{2} \int_{\mathbb{T}^2} \int_{\mathbb{T}^2} G_{\mathbb{T}^2}(x-y)\; u_i (x) \; u_j (y) dx dy
\end{eqnarray}
defined on $BV(\mathbb{T}^2; \{0,1\})$.
In \eqref{energyu}, $u = (u_1, u_2, u_3 )$,  and $u_0 = 1 - \sum_{i=1}^3 u_i $.
The order parameters $u_i, i = 0, 1, 2, 3 $ are defined on $\mathbb{T}^2 = \mathbb{R}^2 / \mathbb{Z}^2=[ - \frac{1}{2}, \frac{1}{2} ]^2$, i.e., the two dimensional flat torus
of unit volume, with periodic boundary conditions. 
Each $u_i$, which represents the relative monomer density, has two preferred states: $u_i = 0$ and $u_i = 1$. 
Case $u_i = 1, i = 1, 2, 3$ or $0$ corresponds to a pure-$A$, pure-$B$, pure-$C$, or pure-$D$ region, respectively. 
Thus, each $u_i=\chi_{\Omega_i}$ with supports $\Omega_i, \ i=0,1,2, 3$ which partition $\mathbb{T}^2$:  $\Omega_i$ are assumed to be mutually disjoint
and  $ \sum_{i=0}^3 u_i =1$ almost everywhere on $\mathbb{T}^2$.

The energy is minimized under three mass or area constraints
\begin{eqnarray} \label{constrain}
 \frac{1}{| \mathbb{T}^2 |}   \int_{\mathbb{T}^2}  u_i  = M_i,  i = 1, 2, 3.
\end{eqnarray}
Here $M_i, i = 1, 2, 3 $ are the area fractions of type-$A$, type-$B$ and type-$C$ regions, respectively. Constraints \eqref{constrain} model the fact
 that, during an experiment, the compositions of the molecules do not change.

The first term in  \eqref{energyu} counts the perimeter of the interfaces:  indeed, for $u_i\in BV(\mathbb{T}^2; \{0,1\})$,
\[ \int_{\mathbb{T}^2} | \nabla u_i |  : = \sup \left \{  \int_{\mathbb{T}^2}  u_i \  \text{div}  \varphi \ dx:  \varphi = (\varphi_1, \varphi_2)\in C^1 (\mathbb{T}^2 ; \mathbb{R}^2), |\varphi(x)| \leq 1 \right \}, \nonumber
\]
 defines the total variation of  the characteristic function $u_i$.  The factor $\frac{1}{2}$ acknowledges that each interface between the phases is counted twice in the sum.

The second part of \eqref{energyu} is the long range interaction energy, associated with the connectivity of sub-chains in the tretrablock copolymer molecule.
The long range interaction coefficients $\gamma_{ij}$ form a symmetric matrix 
\[ \gamma = [\gamma_{ij}]\in \mathbb{R}^{3\times 3}. \]
Here $G_{\mathbb{T}^2}$ is the zero-mean Green's function for $- \triangle$ on $\mathbb{T}^2$ with periodic boundary conditions, satisfying 
\begin{equation}  \label{GLap}
 -\Delta G_{\mathbb{T}^2}(\cdot - y) = \delta(\cdot-y)- 1
 \text{ in } \mathbb{T}^2; 
  \int_{\mathbb{T}^2} {G_{\mathbb{T}^2} (x - y)} dx=0
 \end{equation}
for each $  y\in \mathbb{T}^2$. In two dimensions, the Green's function $G_{\mathbb{T}^2}$ has the local representation 
\begin{equation} \label{G2def}
G_{\mathbb{T}^2}(x - y)= - \frac{1}{2\pi}\log | x-y | + R_{\mathbb{T}^2} (x - y),
\end{equation}
for $|x-y|<\frac12$.
Here $R_{\mathbb{T}^2}\in C^\infty(\mathbb{T}^2)$ is the regular part of the Green's function.

As were the cases for the Ohta-Kawasaki model of diblock copolymers and the Nakazawa-Ohta model of triblock copolymers, nonlocal quaternary systems are of high mathematical interest because
of the diverse patterns which are expected to be observed by its minimizers. 
Just as the diblock copolymer problem or the triblock copolymer problem may be formulated as a nonlocal isoperimetric problem (NLIP) which partitions space into two or three components, respectively, the tetrablock model is an NLIP based on partitions into four disjoint components. 


Here we consider an asymptotic regime where three minority phases have vanishingly small areas. Like \cite{bi1, ablw}, introduce a new parameter $\eta$, controlling the vanishing areas. That is, 
\[   \int_{\mathbb{T}^2}  u_i = \eta^2 M_i \]
for some fixed $M_i$, $i=1, 2, 3$, and rescale $u_i$ as
\begin{eqnarray}
v_{i, \eta}^{} = \frac{ u_i }{\eta^2}, \; i=0,1,2,3,  \; \;\text{ with }   { \int}_{\mathbb{T}^2}  v_{i, \eta}  =  M_i, \; i=1,2,3.
\end{eqnarray}

The matrix $\gamma=[\gamma_{ij}]$ is also scaled, in such a way that both terms in \eqref{energyu} contribute at the same order in $\eta$.  This can be achieved by choosing
\begin{eqnarray}
\gamma_{ij} = \frac{1}{|\log \eta| \eta^3} \Gamma_{ij},    \notag
\end{eqnarray}
with fixed constants $\Gamma_{ij}$. 
Let $v_{\eta} = (v_{1,\eta}, v_{2, \eta}, v_{3, \eta})$. 
Assume $v_\eta$ lies in the space
\begin{equation}\label{Xspace}
  X_\eta:=\left\{ (v_{1,\eta}, v_{2, \eta}, v_{3, \eta}) \ | \  \eta^2 v_{i,\eta}\in BV(\TT; \{0,1\}), \  v_{i,\eta}\, v_{j, \eta} = 0 \ \  a.e. \ \ 1 \leq i < j \leq 3\right\}. 
\end{equation}
For $v_\eta\in X_\eta$, we define
\begin{eqnarray} \label{Eeta}
  E_{\eta}^{} (v_{\eta}) := \frac{1}{\eta} \mathcal{E} ( u ) 
  =
 \frac{\eta}{2} \sum_{i=0}^3 \int_{\mathbb{T}^2} |\nabla v_{i,\eta} | +  \sum_{i,j = 1}^3  \  \frac{  \Gamma_{ij} }{2 |\log \eta| } \int_{\mathbb{T}^2} \int_{\mathbb{T}^2} G_{\mathbb{T}^2}(x - y)  v_{i, \eta}(x) v_{j, \eta}(y) dx dy, \end{eqnarray}
and $E_\eta(v_\eta)=+\infty$ otherwise.
Consider $v_{i, \eta}$ of components of the form
\begin{eqnarray} \label{vi}
v_{i, \eta} = \sum_{k=1}^{\infty} v_{i, \eta}^k,   \text{ and }    v_{i, \eta}^k = \frac{1}{\eta^{2}} \chi_{\Omega_{i,\eta}^k},
\end{eqnarray}
where $ | \Omega_{i,\eta}^k \cap \Omega_{j,\eta}^l | = 0$ when $i \neq j$ or $k \neq l$, each $\Omega_{i,\eta}^k$ is a connected subset of $\Omega_{i,\eta}$, and $\Omega_{i,\eta}= \cup_{k=1}^{\infty} \Omega_{i,\eta}^k$, $i=1, 2, 3$.
Here $|\Omega_{i,\eta}^k \cap \Omega_{j,\eta}^l |$ denotes the area (i.e., the Lebesgue measure) of $ \Omega_{i,\eta}^k \cap \Omega_{j,\eta}^l $.
Assume that each $\Omega_{i,\eta}^k$ does not intersect $\partial \mathbb{T}^2 $, and then we extend $v_{i,\eta}^k$ to zero outside of $\mathbb{T}^2 $. 
%
Let 
\begin{eqnarray} 
z_{i}^k (x):=  \eta^2 v_{i, \eta}^k (\eta x) .
\end{eqnarray} 
Note that $ z_{i}^k (x) : \mathbb{R}^2 \rightarrow \{0, 1\}$ and 
$  \int_{\mathbb{R}^2}  z_{i}^k = \int _{\mathbb{T}^2}  v_{i,\eta}^k $ and $ \int_{\mathbb{R}^2}  | \nabla z_{i}^k | =  \eta \int _{\mathbb{T}^2}  | \nabla v_{i,\eta}^k |.$
Calculation yields
\begin{eqnarray}
  E_{\eta}^{} (v_{\eta}) = \sum_{k=1}^{\infty} \left (      \frac{1}{2}  \sum_{i=0}^3 \int_{\mathbb{R}^2} |\nabla z_{i}^k | +  \sum_{i,j=1}^3  \frac{\Gamma_{ij} }{4\pi}  \left(  \int_{\mathbb{R}^2 }  z_{i}^k \right) \left(  \int_{\mathbb{R}^2 }  z_{j}^k \right)            \right )  + O \left (  \frac{1}{|\log \eta | } \right ).
\end{eqnarray}
Let $m := (m_1, m_2, m_3)$.  
Define
\begin{eqnarray} \label{e0m1}
e_0 (m )  &:= &   \inf  \Bigg \{ \frac{1}{2}  \sum_{i=0}^3 \int_{\mathbb{R}^2} |\nabla z_i | +  \sum_{i,j=1}^3  \frac{\Gamma_{ij} m_i m_j }{4\pi}:  
  z_i \in BV( \mathbb{R}^2; \{0,1\}),  i = 0, 1, 2, 3, \nonumber \\
 &&   \qquad \qquad    \int_{\mathbb{R}^2 } z_i  = m_i, i = 1, 2, 3, \; \ z_0 = 1 - \sum_{i=1}^3 z_i \Bigg \} 
\end{eqnarray}
Since a triple bubble is the unique solution to the three-component isoperimetric problem \cite{proof}
\begin{eqnarray}
 \inf  \Bigg \{  \frac{1}{2}  \sum_{i=0}^3 \int_{\mathbb{R}^2} |\nabla z_i | :  z_i \in BV( \mathbb{R}^2; \{0,1\}),  i = 0, 1, 2, 3, 
 \int_{\mathbb{R}^2 } z_i  = m_i, i = 1, 2, 3, \; \ z_0 = 1 - \sum_{i=1}^3 z_i \Bigg \}, \notag
\end{eqnarray}
\eqref{e0m1} can be simplified as 
\begin{eqnarray} \label{e0m2}
e_0^{} (m)  = p(m_1, m_2, m_3) + \sum_{i, j=1}^3 \frac{\Gamma_{ij} m_i m_j }{4\pi} , 
\end{eqnarray}
where $p(m_1, m_2, m_3) $ is the perimeter of a triple bubble in $\mathbb{R}^2$ when $m_i > 0, i = 1, 2, 3$.
In the case of one $m_i$ equals $0$, that is, $m = (m_1, m_2, 0)$ or $(m_1, 0, m_3)$ or $(0, m_2, m_3) $, a triple bubble degenerates to a double bubble and \eqref{e0m2} becomes 
\begin{eqnarray} \label{e0d}
\text{ } e_0(m) = e_0^d(m_i, m_j) :=  p^d (m_i, m_j) + \frac{\Gamma_{ii} (m_i)^2 }{4\pi} + \frac{\Gamma_{jj} (m_j)^2 }{4\pi} + \frac{\Gamma_{ij} m_i m_j }{2\pi},
\end{eqnarray}
where $1 \leq i,j \leq 3, i \neq j $, and $p^d (m_i, m_j) $ is the perimeter of a double bubble in $\mathbb{R}^2$.
In the case of two $m_i$ equal $0$, that is, $m = (m_1, 0, 0)$ or $(0, m_2, 0)$ or $(0, 0, m_3) $, a triple bubble degenerates to a single bubble and \eqref{e0m2} becomes 
\begin{eqnarray} \label{e0s}
\text{ } e_0(m)  = e_0^s(m_i) :=  2\sqrt{\pi m_i}  +  \frac{\Gamma_{ii} (m_i)^2 }{4\pi}.
\end{eqnarray}
Thus, we expect that minimizers of $E_\eta$ in two dimensions will always consist of single, double or triple bubbles (or mixture);  no other shapes are expected for the components of $\Omega_{i,\eta}$.  The spatial distribution of the single, double or triple bubbles on $\TT$ should be determined by the higher order terms of $E_\eta$ in a more detailed energy expansion.

To further discuss how the total masses $M=(M_1,M_2, M_3)$ are to be divided, define
  \begin{eqnarray}  \label{mine0}
 \overline{e_0 }(M ) := \inf \left\{ \sum_{k=1}^{\infty} e_0 (m^k ) :  m^k = (m_1^k, m_2^k, m_3^k ),  \ m_i^k \geq 0,\ \sum_{k=1}^{\infty}  m_i^k = M_i, i = 1, 2, 3 \right\}.
 \label{e0bar}
 \end{eqnarray}
The main result of this paper concerns the minimizers of $ \overline{e_0 }(M )$.
\begin{theorem} \label{coexistence}
 (coexistence) 
There exist global minimizers with arbitrarily large number of single, double or triple bubbles. 
That is, given $N_i$, $i=1,2,3$, there exist parameters $M_i$ and $\Gamma_{ij}$ such that
all global minimizers have at least $N_1$ (resp. $N_2$, $N_3$) triple bubbles (resp.  double, single bubbles).
 \end{theorem}

These are proven via delicate comparison arguments based on the geometry of triple bubbles, in Section~\ref{section 2}.
In Section~\ref{section 3} we discuss minimizers at $\eta$ level.




While mathematical study of nonlocal quaternary systems 
is relatively limited,
recently there has been progress in mathematical analysis of nonlocal ternary systems:
\begin{eqnarray}  \label{energyT}
\mathcal{E} (u) :=  \frac{1}{2} \sum_{i=0}^2 \int_{\mathbb{T}^2} |\nabla u_i | +  \sum_{i,j = 1}^2  \frac{\gamma_{ij}}{2} \int_{\mathbb{T}^2} \int_{\mathbb{T}^2} G_{\mathbb{T}^2}(x-y)\; u_i (x) \; u_j (y) dx dy.
\end{eqnarray}
For ternary systems, one-dimensional stationary points to the Euler-Lagrange equations of \eqref{energyT} were found in \cite{lameRW, blendCR}. 
Two and three dimensional stationary configurations were studied recently in \cite{double, doubleAs, stationary, disc, evolutionTer}.
There are also plenty of mathematical work on nonlocal binary systems. Much early work
concentrated on the diffuse interface Ohta-Kawasaki density functional theory for diblock copolymers \cite{equilibrium, nishiura, onDerivation},
\begin{eqnarray} \label{energyB}
\mathcal{E}^{} (u) :=  \int_{\mathbb{T}^n} |\nabla u | +    \gamma \int_{\mathbb{T}^n} \int_{\mathbb{T}^n} G_{\mathbb{T}^n}(x-y)\; u (x) \; u (y) dx dy,
\end{eqnarray}
with a single mass or volume constraint. The dynamics for a gradient flow for \eqref{energyB} with small volume fraction were developed in \cite{hnr, gc}. 
All stationary solutions to the Euler-Lagrange equation of \eqref{energyB} in one dimension were known to be local minimizers \cite{miniRW}, and
many stationary points in two and three dimensions have been found that match the morphological phases in diblock copolymers \cite{oshita, many, ihsan, Julin3, cristoferi, afjm}.
The sharp interface nonlocal isoperimetric problems have been the object of great interest, both for applications and for their connection to problems of minimal or constant curvature surfaces.
Global minimizers of \eqref{energyB}, and the related Gamow's Liquid Drop model describing atomic nuclei, were studied in \cite{otto, muratov, bi1, st, GMSdensity, knupfer1, knupfer2, Julin, ms, fl} for various parameter ranges.
Variants of the Gamow's liquid drop model with background potential or with an anisotropic surface energy replacing the perimeter, are studied in \cite{ABCT1,luotto, cnt}.
Higher dimensions are considered in \cite{BC, cisp}.
Applications of the second variation of \eqref{energyB} and its connections to minimality and $\Gamma$-convergence are to be found in \cite{cs,afm,Julin2}.
 The bifurcation from spherical, cylindrical and lamellar shapes with Yukawa instead of Coulomb interaction has been done in \cite{fall}.
Blends of diblock copolymers and nanoparticles \cite{nano, ABCT2} and blends of diblock copolymers and homopolymers are also studied by \cite{BK,blendCR}.
Extension of the local perimeter term to nonlocal $s$-perimeters is studied in \cite{figalli}.

\subsection*{Acknowledgements}  Stanley Alama, Lia Bronsard, and Xinyang Lu gratefully acknowledge the support of the Natural Science and Engineering Research Council (Canada) through the Discovery Grants program.


\section{Geometric properties of global minimizers} \label{section 2}

In this section we analyze the geometric properties of minimizers of $\overline{e_0}(M)$. One of the key difficulties is that we lack an explicit formula for the perimeter of triple bubbles. Recall that a triple bubble is a group of three adjacent sets
bounded by six circular arcs; see Figure 1.

\begin{figure}[h!]\label{tri1}                                       
\centering
\includegraphics[height=4cm]{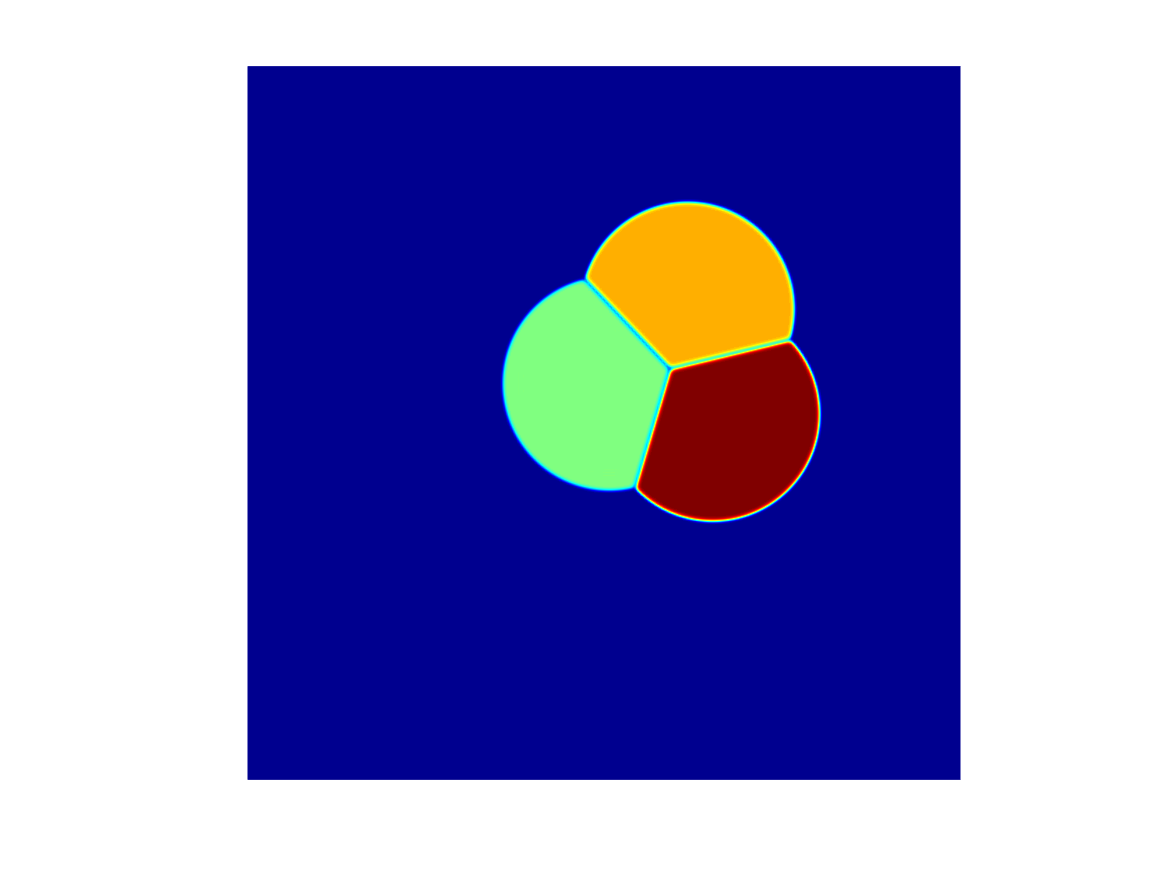} 
\includegraphics[height=4cm]{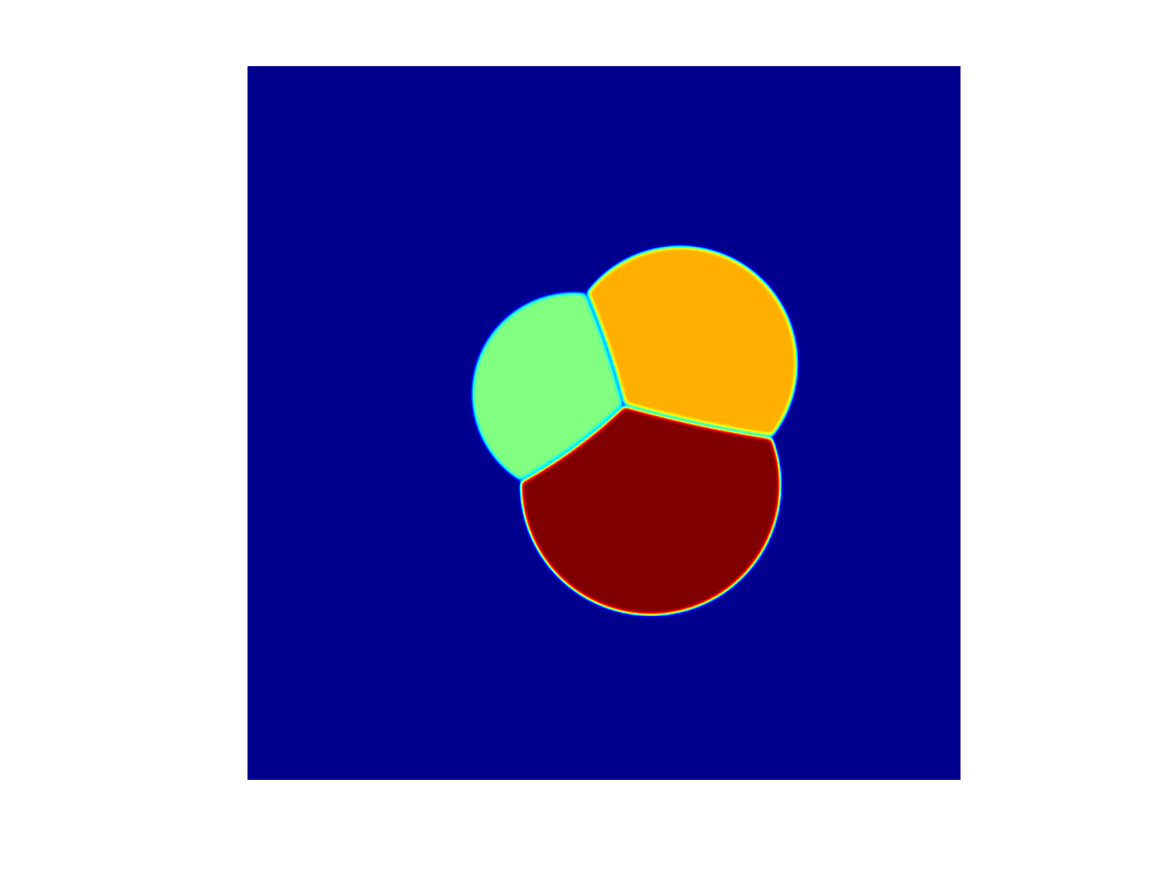}
\caption{
Symmetric (left) and asymmetric (right) triple bubbles.}
\end{figure}

Figure \ref{tribub2} shows a triple bubble $\Omega=(\Omega_1,\Omega_2,\Omega_3)$, where $\Omega_1$ is bounded by three circular arcs $C_1$, $
C_4$ and $C_5$ of radii $r_1$, $r_4$ and $r_5$; $\Omega_2$ is bounded by three circular arcs $C_2$, $
C_4$ and $C_6$ of radii $r_2$, $r_4$ and $r_6$;  $\Omega_3$ is bounded by three circular arcs $C_3$, $
C_5$ and $C_6$ of radii $r_3$, $r_5$ and $r_6$.
\begin{figure}[h!]                                      
\centering
\includegraphics[height=9cm]{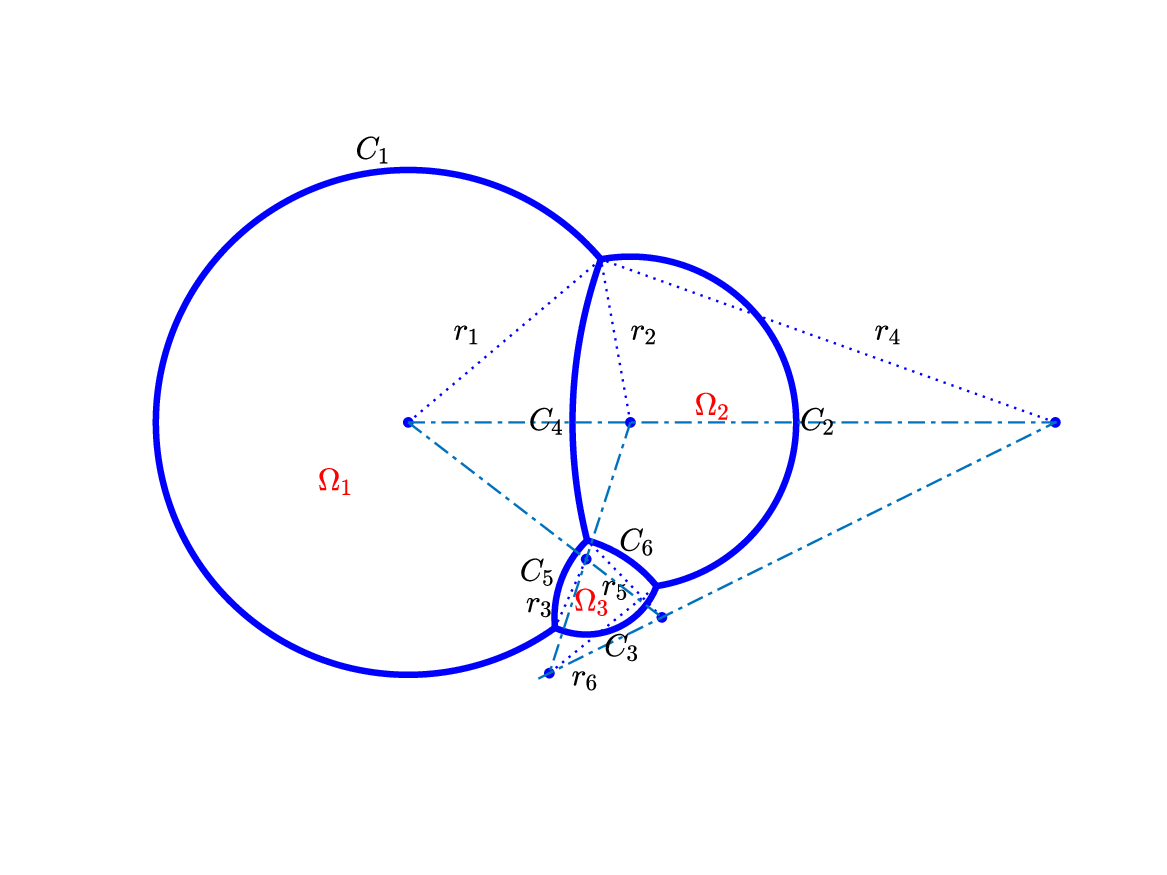} 
\caption{An asymmetric triple bubble with radii $r_i, 1 \leq i \leq 6$.}
 \label{tribub2}    
\end{figure}
In the case of a double bubble, the radii of curvature of the three circular arcs are related by a reciprocal relation \cite{ablw}.
For a triple bubble, the reciprocal relation holds for each pair of bubbles \cite{isenberg}, that is, 
\begin{eqnarray} \label{radiiCon}
\frac{1}{r_4} + \frac{1}{r_1} - \frac{1}{r_2} = 0,  \qquad
\frac{1}{r_5} + \frac{1}{r_1} - \frac{1}{r_3} = 0, \qquad
\frac{1}{r_6} + \frac{1}{r_2} - \frac{1}{r_3} = 0.\qquad
\end{eqnarray}
The area constraints $ | \Omega_i | = m_i, i = 1, 2, 3$ can be expressed as
\begin{eqnarray}
\frac{1}{2} r_1^2(\theta_1 - \sin \theta_1) - \frac{1}{2} r_4^2( \theta_4 - \sin \theta_4) - \frac{1}{2} r_5^2 (\theta_5- \sin \theta_5) + \frac{1}{2} h_4 h_5 \sin s_4 &=& m_1,
 \\
\frac{1}{2} r_2^2( \theta_2 - \sin \theta_2) - \frac{1}{2} r_6^2( \theta_6- \sin \theta_6)  + \frac{1}{2} r_4^2( \theta_4 - \sin \theta_4)  + \frac{1}{2} h_4 h_6 \sin s_5 &=& m_2, \\
\frac{1}{2} r_3^2( \theta_3 - \sin \theta_3) + \frac{1}{2} r_5^2( \theta_5 - \sin \theta_5)  + \frac{1}{2} r_6^2( \theta_6 - \sin \theta_6)  + \frac{1}{2} h_5 h_6 \sin s_6& =& m_3,
\end{eqnarray}
where $\theta_i$ and $h_i$ are the angle and chord associated with the circular arc $C_i$, respectively, $s_4$ is the angle between chords $h_4$ and $h_5$, $s_5$ is the angle between chords $h_4$ and $h_6$, and 
$s_6$ is the angle between chords $h_5$ and $h_6$.

The four points where the three arcs meet are termed triple junction points. At triple junction points, the three arcs meet at $120$ degree angle.
Let $O_i$ denotes the center corresponding to the circular arc $C_i$.
The radii conditions \eqref{radiiCon} and the 120 degree angle conditions imply that centers 
$O_1, O_2, O_4$ lie on a straight line, centers $O_1, O_3, O_5$ lie on a straight line, centers $O_2, O_3, O_6$ lie on a straight line, and centers $O_4, O_5, O_6$ lie on a straight line.

\begin{lemma}\label{perimeter derivative}
It holds
\[\frac{\partial}{\partial m_i} p(m_1,m_2,m_3) = \frac{1}{r_i},\qquad i=1,2,3, \]
where $r_i=r_i(m_1,m_2,m_3)$ denotes the radius of the lobe of type $i$ constituent; see Figure \ref{tribub2}.
\end{lemma}

\begin{proof}
The proof is similar to the double bubble case, and uses similar constructions \cite{ablw}.
\end{proof}

\begin{proposition}
\label{min mass}
	Let the total masses $M_i$ and interaction coefficients $\ggm_{ij}$ be given.
	Consider a local minimizer of $\overline{e_0}(M)$, and denote by $B^k=(m_{1}^k,m_{2}^k,m_{3}^k)$ its bubbles. Then 
	\[ \min_k m_{i}^k \ge m_{i}^{-},\qquad i=1,2,3\]
	where
	\[m_{i}^{-}  \ge c_1 \min_k (r_{i}^k)^2 \ge c_1 \min\bigg\{ \bigg[\frac{1}{\pi} \sum_{j=1}^3 \Gamma_{ij}M_j\bigg]^{-2}, \bigg[\frac{C_2 M_i}{ 4 c_2\inf E_0 }\bigg]^2 \bigg\},\qquad i=1,2,3\]
	is some lower bound depending only on $M_i$ and $\ggm_{ij}$. Here $r_{i}^k $ stands for the radius of the lobe with mass $m_{i}^k $, and $c_1,c_2,C_2>0$ are universal geometric
	constants (see \eqref{mass estimate} and \eqref{case 2 estimate} below).
	
	Consequently, the number of bubbles of type $i$ constituent does not exceed $M_i/m_{i}^{-}$.
\end{proposition}


\begin{proof}
Based on \eqref{e0m2} and the stability condition
	\begin{equation}
	\frac{\partial  }{\partial m_{i}^k}  \left( \sum_{k=1}^{\infty}  e_0 (m^k)  \right ) = \frac{1}{2\pi} \sum_{j=1}^3 \Gamma_{ij}m_{j}^k +  \frac{\partial p}{\partial m_{i}^k}
	= \frac{1}{2\pi} \sum_{j=1}^3 \Gamma_{ij}m_{j}^k+ \frac{1}{r_{i}^k} =\la_i ,\qquad i=1,2,3,
\label{lagrange multiplier}
	\end{equation}
where 
	 $\la_i$ is the Lagrange multiplier, clearly independent of $k$.
	 Recall that there exist universal constants $c_2>c_1>0$ such that
\begin{equation}
	  c_1 (r_{j}^k)^2 \le m_{j}^k \le c_2 (r_{j}^k)^2 ,\qquad \fal j,k. \label{mass estimate}
\end{equation}
{
	Case A.
	We first assume that $\min_k r_{i}^k$ and $\max_k r_{i}^k$ exist.
	By \eqref{lagrange multiplier}, both of them are positive.
	 Let}
	 \[r_{i}^1=r_{i}^{\min}:=\min_k r_{i}^k,\qquad  r_{i}^2 =r_{i}^{\max}:=\max_k r_{i}^k.\]
	 The aim is to bound $r_{i}^1$ from below. By \eqref{lagrange multiplier} we get
	 \[\la_i = \frac{1}{2\pi} \sum_{j=1}^3 \Gamma_{ij}m_{j}^k  + \frac{1}{r_{i}^k } \qquad \fal k\ge1.\]
	 	In particular, taking $k=1,2$,
	 	\[\la_i = \frac{1}{2\pi} \sum_{j=1}^3 \Gamma_{ij}m_{j}^1 + \frac{1}{r_{i}^1} 
=	 	 \frac{1}{2\pi} \sum_{j=1}^3 \Gamma_{ij}m_{j}^2 + \frac{1}{r_{i}^2},\]
hence
\begin{equation}
\frac{1}{r_{i}^1}  - \frac{1}{r_{i}^2} = \frac{1}{2\pi} \sum_{j=1}^3 \Gamma_{ij}[m_{j}^2-m_{j}^1].
\label{radii difference}
\end{equation}
Let 
\begin{eqnarray}
\al:=r_{i}^2/r_{i}^1 \ge 1.      \label{defalpha}
\end{eqnarray} Two possibilities arise now.

\medskip

{\em Subcase A.1}: $\al \ge 2$. In this case, \eqref{radii difference} read
 	\[\frac{(\al-1)}{\al r_{i}^1} =\frac{1}{r_{i}^1}  - \frac{1}{r_{i}^2} 
 	= \frac{1}{2\pi} \sum_{j=1}^3 \Gamma_{ij}[m_{j}^2-m_{j}^1] \le 
 	\frac{1}{2\pi} \sum_{j=1}^3 \Gamma_{ij}M_j, \]
 	hence 
 	\[\frac{1}{r_{i}^1} \le \frac{\al}{\al-1} \frac{1}{2\pi} \sum_{j=1}^3 \Gamma_{ij}M_j 
 	\le \frac{1}{\pi} \sum_{j=1}^3 \Gamma_{ij}M_j,\]
where the last inequality is due to $\al \ge 2$, which gives $\al/(\al-1) \le 2$. Thus in this case we get
\begin{equation}
r_{i}^1 \ge \bigg[\frac{1}{\pi} \sum_{j=1}^3 \Gamma_{ij}M_j\bigg]^{-1}.
\label{case 1 estimate}
\end{equation}

\medskip

{\em Subcase A.2}: $\al \le 2$. In this case, \eqref{radii difference} alone is not sufficient to give a lower bound
on $r_{i}^1$. We must use the energy term. Note that, in this case, all radii $r_{i}^k $ satisfy
$r_{i}^1  \le r_{i}^k  \le 2r_{i}^1 $. By \eqref{mass estimate}, we have at least
\begin{equation}
\frac{M_i}{\sup_k m_{i}^k } \ge \frac{M_i}{c_2\sup_k (r_{i}^k)^2} = \frac{M_i}{c_2 (r_{i}^2)^2}
\ge \frac{M_i}{ 4 c_2 ( r_{i}^1)^2} \label{mass radii comparison}
\end{equation}
lobes of type $i$ constituent.

The total energy is not smaller than the sum of perimeters, which is, in turn, bounded via the isoperimetric inequality by
\begin{equation}
 \sum_{k\geq 1} p(m_{1}^k ,m_{2}^k ,m_{3}^k )  \ge C_1\sum_{k\geq 1} \sqrt{m_{i}^k } \overset{\eqref{mass estimate}}{\ge} C_2\sum_{k\geq 1} r_{i}^k  \ge \frac{C_2 M_i}{4 c_2 r_{i}^1 }
 \label{case 2 estimate}
\end{equation}
for some universal geometric constants $C_1,C_2>0$.  The last inequality follows from the fact that
there are at
least  $\frac{M_i}{4 c_2 (r_{i}^1)^2}$ lobes of type $i$ constituent. Thus \eqref{case 2 estimate}
gives 
$  r_{i}^1 \ge \frac{C_2 M_i}{ 4 c_2\inf E_0 }$. Combining with \eqref{case 1 estimate} and
\eqref{mass estimate} concludes the proof for this case.

\medskip

 Case B. We now remove the assumption that both
 $\min_k r_{i}^k$ and $\max_k r_{i}^k$ exist. Consider sequences
\[r_{i}^{1,\vep(k)}\searrow r_{i}^{\min}:=\inf_k r_{i}^k\ge 0,\qquad  r_{i}^{2,\vep(k)} \nearrow r_{i}^{\max}:=\sup_k r_{i}^k,\]
and denote by $m_{i}^{j,\vep(k) }$, $j=1,2$, the associated masses. Without loss of generality,
we can ensure $r_{i}^{1,\vep(k)}\le r_{i}^{2,\vep(k)}$.
Repeat the same arguments as Case A with replacing $r_i^j, m_i^j$ by $r_i^{j, \epsilon(k)}, m_i^{j,\epsilon(k)}$ respectively . Similarly let 
\[ \al(k) :=r_{i}^{2,\vep(k)}/r_{i}^{1,\vep(k)}\ge 1,  \]
and again we have two cases.

{\em Subcase B.1:} $\inf_k \al(k)\ge 2$. Similar arguments yield 
\begin{equation}
r_{i}^{1,\vep(k)} \ge \bigg[\frac{1}{\pi} \sum_{j=1}^3 \Gamma_{ij}M_j\bigg]^{-1}.
\label{case 1 estimate vep}
\end{equation}

{\em Subcase B.2}: $\inf_k \al(k)\le 2$. Note that, in this case, all radii $r_{i}^k $ satisfy
$r_{i}^{1,\vep(k)}-\dt  \le r_{i}^k  \le 2r_{i}^{1,\vep(k)}+\dt $, for some
small $\dt$ which vanishes as $\vep(k)\to 0$. Similarly we get
\begin{eqnarray}
 r_{i}^{1,\vep(k)} -\dt \ge \frac{C_2 M_i}{ 4 c_2\inf E_0 }.
 \end{eqnarray}
 Taking the limit
$\vep(k)\to 0$, which in turn forces $\dt\to 0$, concludes the proof. 

	\end{proof}

	To prove Theorem \ref{coexistence}, we need first to get an upper bound for the mass of lobes.
	
	\begin{lemma}\label{max mass}
			Let the total masses $M_i$ and interaction coefficients $\ggm_{ij}$ be given.
			Consider a local minimizer of $\overline{e_0}(M)$, and denote by $B^k=(m_{1}^k,m_{2}^k,m_{3}^k)$ its bubbles. Then 
			\[ \sup_k m_{i}^k \le m_{i}^{+}:= \frac{8 \pi}{\ggm_{ii}^{2/3}},\qquad i=1,2,3.\]
	\end{lemma}
	
	\begin{proof}
		We write the proof for $i=1$ only, since $i=2,3$ are dealt with analogously.
Consider a bubble $B$ with lobes of masses
$m_{1},m_{2},m_{3}$, and assume $m_{1}$ is bigger than $m_2$ and $m_3$. The proof follows by comparing its energy with that of 
the two bubbles $B'$ and $B''$, where $B'$ has lobes of masses $m_{1}/2,m_{2},m_{3}$, and $B''$ is a single bubble of mass
$m_1/2$. That is, we compare the energy of the original minimizer with that of the competitor obtained
by replacing $B$ with $B'\cup B''$.
The former is 
\[ E_B:=\frac{1}{4\pi}\sum_{i,j=1}^3 \ggm_{ij} m_im_j +p(m_1,m_2,m_3) , \]
while the latter is $E_{B'}+E_{B''}$, where
\begin{align}
E_{B'} &:=
\frac{1}{8\pi}\sum_{j=1}^3 \ggm_{1j} m_1m_j+ \frac{1}{4\pi}\sum_{i,j=2}^3 \ggm_{ij} m_im_j +  p(m_1/2,m_2,m_3) , \\
E_{B''} &:= \frac{\ggm_{11}m_1^2}{16\pi} +\sqrt{2\pi m_1}.\label{energy diff max mass}
\end{align}
By the minimality of $B$, we need
\begin{align}
 0 &\ge E_B - E_{B'}-E_{B''}  = \frac{1}{8\pi}\sum_{j=1}^3 \ggm_{1j} m_1m_j -\frac{\ggm_{11}m_1^2}{16\pi}  + \Delta_{Per} \notag \\
 \label{bound1}
 & = \frac{\ggm_{11}m_1^2}{16\pi}  + \frac{1}{8\pi}\sum_{j=2}^3 \ggm_{1j} m_1m_j + \Delta_{Per},
\end{align}
where
\begin{eqnarray}
 \Delta_{Per}:= p(m_1,m_2,m_3)- p(m_1/2,m_2,m_3)  - \sqrt{2\pi m_1}. \notag
\end{eqnarray}
Lemma \ref{perimeter derivative} gives
\begin{align*}
 p(m_1,m_2,m_3)- p(m_1/2,m_2,m_3)  & = \int_{m_1/2}^{m_1} \frac{\pd p( m,m_2,m_3) }{\pd m}\d m 
 =\int_{m_1/2}^{m_1} \frac{1 }{r_1}\d m \ge 0.
\end{align*}
Thus \eqref{bound1} reads
\begin{align*}
 0 &\ge \frac{\ggm_{11}m_1^2}{16\pi}  + \frac{1}{8\pi}\sum_{j=2}^3 \ggm_{1j} m_1m_j +\Delta_{Per}
 \ge \frac{\ggm_{11} m_1^2}{16\pi} 
  -  \sqrt{2\pi m_1},
\end{align*}
which concludes the proof.
	\end{proof}


	In the proof of Theorem \ref{coexistence}, we will take $\ggm_{ij}=0$ if $i\neq j$.
Since we will choose the total masses $M_i $ one after another, it is important that the 
 the lower bound
	$m_{i}^{-}$ can be made independent of $M_j$, $j\neq i$. The next result shows that this is really the case.
	
	\begin{lemma}\label{independent mass}
		Under the hypotheses $\ggm_{ij}=0$ if $i\neq j$, then 
		$m_{i}^{-}$ can be made independent of $M_j$, $j\neq i$.
	\end{lemma}
	
	\begin{proof}
	         Here we only discuss the scenario that $\min_k r_{i}^k$ and $\max_k r_{i}^k$ exist. For the case of nonexistence, the proof is similar to Case B in Proposition \ref{min mass}.
		Let $\al$ be the ratio defined in \eqref{defalpha}. 
		There are two cases.
		
		\medskip
		
		{\em Case} $\al \ge 2$. Then \eqref{case 1 estimate} combined with \eqref{mass estimate} give directly
		\[ \min_{k} m_{i}^k  \ge c_1  \min_{k} ( r_{i}^k )^2 \ge c_1   \bigg[\frac{1}{\pi} \sum_{j=1}^3 \Gamma_{ij}M_j\bigg]^{-2} =   \frac{c_1\pi^2}{ \Gamma_{ii}^2 M_i^2} . \]
		%
		\medskip
		
		{\em Case} $\al \le 2$. By \eqref{mass estimate} we get
		\[ \frac{m_{i}^{+}}{m_{i}^{-} } \le \frac{c_2 (r_{i}^{+})^2}{c_1 (r_{i}^{-})^2} \le \frac{4c_2}{c_1}.  \]
		We analyze what happens if we remove a certain number, say $n$, of bubbles/lobes of type $i$ constituent, and combine them
		into a single bubble. Removing one such bubble/lobe decreases the energy at least by an amount
		\[ c\sqrt{m_i} + \frac{1}{4\pi}\sum_{j=1}^{3} \ggm_{ij} m_im_j =
		c\sqrt{m_i} + \frac{\ggm_{ii} m_i^2}{4\pi} , \] 
		for some geometric constant $c>0$, hence removing $n$ such bubbles/lobes decreases the energy by
		\[E^-:=\sum_{k= 1}^n c\sqrt{m_{i}^k } + \sum_{k= 1}^n\frac{\ggm_{ii} (m_i^k)^2}{4\pi}. \] 
		Adding a single bubble of mass $\sum_{k= 1}^n m_{i}^k $ increases the energy by
		\[ E^+:= \frac{\ggm_{ii}}{4\pi}\bigg[\sum_{k= 1}^n m_{i}^k \bigg]^2 + 2\sqrt{\pi \sum_{k= 1}^n m_{i}^k }. \]
		By the minimality of our configuration, we need 
		\begin{align}
		0 \ge E^- - E^+ &=\Delta_{Per} 
		+  \frac{\ggm_{ii}}{4\pi} \bigg\{ \sum_{k= 1}^n (m_i^k)^2-\bigg[\sum_{k= 1}^n m_{i}^k \bigg]^2\bigg\} ,\qquad
		\Delta_{Per}:=c\sum_{k= 1}^n\sqrt{m_{i}^k } -2\sqrt{\pi \sum_{k= 1}^n m_{i}^k },\notag\\
		&=\Delta_{Per} - \frac{\ggm_{ii}}{4\pi} \sum_{k\neq h}
		m_{i}^k m_{i}^h 
		\ge\Delta_{Per}
		- \frac{ n(n-1) \Gamma_{ii} }{4\pi} (m_{i}^{+})^2.\label{minimality upper bound number components}
		\end{align} 
		On the other hand, 
		\begin{align*}
		\Delta_{Per}& =
		\frac{c^2(\sum_{k= 1}^n\sqrt{m_{i}^k })^2 -  4\pi \sum_{k= 1}^n m_{i}^k }{ c\sum_{k= 1}^n\sqrt{m_{i}^k } +2\sqrt{\pi \sum_{k= 1}^n m_{i}^k }}
		=\frac{(c^2-4\pi )\sum_{k= 1}^n m_{i}^k  + c^2\sum_{k\neq h}\sqrt{m_{i}^k  m_{i}^h } }{ c\sum_{k= 1}^n\sqrt{m_{i}^k } +2\sqrt{\pi \sum_{k= 1}^n m_{i}^k }}.
		\end{align*}
		Choose $c$ such that $c^2-4\pi$ be negative. The numerator is bounded from below by
		\begin{align*}
		(c^2-4\pi )\sum_{k= 1}^n m_{i}^k  + c^2\sum_{k\neq h}\sqrt{m_{i}^k m_{i}^h } &\ge
		(c^2-4\pi ) n m_{i}^+   + c^2 n(n-1) m_{i}^- \\
		&\ge\bigg[ (c^2-4\pi ) +  \frac{c^2 (n-1)c_1}{4 c_2} \bigg] n m_{i}^+,
		\end{align*}
		while the denominator is bounded from above by
		\[c\sum_{k= 1}^n\sqrt{m_{i}^k } +2\sqrt{\pi \sum_{k= 1}^n m_{i}^k } \le [cn+2\sqrt{n\pi }] \sqrt{m_{i}^+} . \]
		Therefore,
		\[\Delta_{Per}\ge \bigg[ (c^2-4\pi ) +  \frac{c^2 (n-1)c_1}{4 c_2} \bigg]\cdot [c+2\sqrt{\pi/n }]^{-1}  \sqrt{m_{i}^+},\]
		and \eqref{minimality upper bound number components} now reads
		\begin{equation}
		0\ge \bigg[ (c^2-4\pi ) +  \frac{c^2 (n-1)c_1}{4 c_2} \bigg]\cdot [c+2\sqrt{\pi/n }]^{-1}  \sqrt{m_{i}^+} -
		\frac{ n(n-1) \Gamma_{ii} }{4 \pi} (m_{i}^+)^2.\label{minimality upper bound number components 2}
		\end{equation}
		This inequality must hold for all $n\le N_i$, defined as the number of bubbles/lobes of type
		$i$ constituent. 
		If for all such admissible $n$, we have
		\[(c^2-4\pi ) +  \frac{c^2 (n-1)c_1}{4 c_2} <0,\]
		then 
		\[N_i \le 1+ \frac{4 c_2	(4\pi -c^2)}{c^2 c_1},\]
		where the right hand side term is clearly independent of $M_j$, $j\neq i$.
		Now consider the case
		$N_i > 1+ \frac{4 c_2	(4\pi -c^2)}{c^2 c_1}$
		Until now, we have not chosen $n$ yet. So choose 
		\[n:= \begin{cases}
		\bigg\lceil 1+ \frac{4 c_2	(4\pi -c^2)}{c^2 c_1}  \bigg\rceil, & \text{if } \frac{4 c_2	(4\pi -c^2)}{c^2 c_1}  \notin \mathbb{N},\\
		2+ \frac{4 c_2	(4\pi -c^2)}{c^2 c_1} , & \text{if }  \frac{4 c_2	(4\pi -c^2)}{c^2 c_1}  \in \mathbb{N}.
		\end{cases}\]
		Such $n$ ensures $(c^2-4\pi ) +  \frac{c^2 (n-1)c_1}{4 c_2} >0$, and \eqref{minimality upper bound number components 2}
		now gives
		\[ (m_{i}^+)^{3/2} \ge \frac{4 \pi}{n(n-1) \Gamma_{ii} } \bigg[ (c^2-4\pi ) +  \frac{c^2 (n-1)c_1}{4 c_2} \bigg]\cdot [c+2\sqrt{\pi/n }]^{-1} ,\]
		with the last term being purely geometric constant. Recalling that $\frac{m_{i}^{+}}{m_{i}^{-} } \le \frac{4c_2}{c_1}$ concludes the proof.
	\end{proof}

	\begin{proof}[\bf{Proof of Theorem~\ref{coexistence}}]  
		Let $\ggm$ be a diagonal matrix, i.e. $\ggm_{ij}=0$ if $i\neq j$. Based on \eqref{e0m2} - \eqref{e0s}:
		\begin{enumerate}
			 \item there cannot be two single bubbles of different type constituent. This because
			merging them into a double bubble is energetically favorable.
			
			\item there cannot be a single bubble of type $i$ constituent, and a double bubble
			without lobes of type $i$ constituent. This again because
			merging them into a triple bubble is energetically favorable.
		\end{enumerate}

	Note the number of bubbles/lobes
	of type $i$ constituent is between $M_i/m_{i}^+$ and $M_i/m_{i}^-$.
	It is of crucial importance that the bounds $m_{i}^{\pm}$ are independent of $M_j$ (see Lemmas
	\ref{max mass} and \ref{independent mass}), since this allows $M_i/m_{i}^{\pm}$ to be completely
	independent of $M_j/m_{j}^{\pm}$, $i\neq j$.
	Thus by choosing suitably large $M_i$, we can ensure to have arbitrarily large $M_i/m_{i}^+$, which ensures
	an arbitrarily large numbers of single/double/triple
	bubbles. Indeed, Choosing $M_2,M_3\ge M_1$, with $M_1/m_{1}^+ >N_1$ ensures the presence of at least $N_1$
	triple bubbles. On the other hand, the number of lobes of type $1$ constituent cannot exceed $M_1/m_{1}^-$.
	
	Thus, the masses of type $i=2,3$ constituents in triple bubbles is at most $m_{i}^+ M_1/m_{1}^-$. Then choosing sufficiently large
	$M_3\ge M_2$, such that $M_i/m_{i}^+ >m_{i}^+ N_d +  m_{i}^+ M_1/m_{1}^-$, $i=2,3$, ensure the presence of at least $N_2$
	double bubbles. Finally, noting again that the number of double bubbles cannot exceed $(M_2 -m_{i}^+ M_1/m_{1}^-) /m_{2}^-  $,
so the mass of type $3$ constituent in double bubbles cannot exceed $m_{3}^+ (M_2 -m_{i}^+ M_1/m_{1}^- ) /m_{2}^-  $, we can finally choose large $M_3$ to ensure 
	the presence of at least $N_1$ single bubbles. Clearly, such constructions leads to all single bubble being of type $3$ constituent, and all double
	bubbles of types $2$ and $3$ constituents. The proof is thus complete.

%
%
	\end{proof}

%


\section{Minimizers at $\eta$ level}\label{section 3}

Now we investigate the behavior of minimizers of $E_\eta$, $\eta>0$. The next result shows that the number of connected components
is again bounded from above.

\begin{lemma}\label{bounded number of components eta}
Let $M_i$ and $\ggm_{ij}$ be given.
Consider a sequence of minimizers $v_\eta=(v_{1, \eta} ,v_{2, \eta}, v_{3, \eta} )$ of $E_\eta$, subject to the mass constraint
$\int_{\mathbb{T}^2} v_{i, \eta } = M_i$. Denote by 
\[v_0=(v_{1,0},v_{2,0},v_{3,0}),\qquad v_{i, 0} = \sum_{k=1}^K m_{i}^k \dt_{x^{k}},\quad \sum_{i=1}^3 m_{i}^k >0, \]
the limit measure such that $v_\eta \rightharpoonup v_0$, and $E_0(v_0)=\min E_0$. Then, for all sufficiently small
$\eta>0$, the number of connected components
of each $v_{i, \eta}$ does not exceed $\sh \{k:m_{i}^k >0\}$, i.e. the number of nonzero masses $m_{i}^k$.
\end{lemma}

\begin{proof}
	The proof is completely analogous to the ternary case \cite{ablw}.
\end{proof}

The next result shows that, both at $\eta>0$ and $\eta =0$ levels, two connected components cannot be too close to each other.
\begin{proposition}
Let $M_i$ and $\ggm_{ij}>0$ be given.
Consider a minimizer $v_\eta=(v_{1,\eta} ,v_{2,\eta}, v_{3,\eta} )$ of $E_\eta$, subject to the mass constraint
$\int_{\mathbb{T}^2} v_{i, \eta } = M_i$. Then, denoting by $\om_{i,\eta}^k$ the connected components of $v_{i,\eta}$,
it holds
\[ \inf_{ x\in \om_{i,\eta}^k,\ y\in \om_{i,\eta}^h, \ k\neq h } |x-y| \ge \dt >0, \]
where $\dt$ is some computable constant depending only on $M_i$ and $\ggm_{ij}$.
\label{minimum distance between components}
\end{proposition}


\begin{proof}
	Consider a sequence of minimizers $v_{\eta}$, such that $E_\eta(v_{\eta})=\inf E_\eta$,
	and denote by $v_0$ the limit function, corresponding to the $\eta=0$ level.  
	
	Denote by $\om_{i,\eta}^k $ the lobes of the connected components of $v_{i,\eta}$,
	and by $K_i$ the number of connected components of $v_{i,\eta}$, $i=1,2,3$.
	The main idea of the proof is that, if two connected components are too close to each other, then the logarithmic part in the
	nonlocal interaction term would be too large. 
	
	\medskip
	
	{\em Step 1. Finding an uniform lower bound on the masses.}
	Lemma \ref{bounded number of components eta} gives that, for all sufficiently
small $\eta$, the number of connected
components of $v_{i,\eta}$ does not exceed $\# \{k:m_{i}^k >0\}$. 
Therefore, 
	\[ m_{i,\eta}^+ :=\sup_{k} |\om_{i,\eta}^k| \ge \frac{M_i}{\# \{k:m_{i}^k >0\}}>0. \]
Note the right hand side is crucially independent of $\eta$. 
In particular, there exists a lobe with radius $r_{i,\eta}^+\ge c_1 m_{i,\eta}^+$, for some
geometric constant $c_1$. Since each $v_\eta$
is a minimizer, the Lagrange multiplier satisfies 
\[ \Lambda \le \frac{1}{r_{i,\eta}^+} +\sum_{k=1}^3\frac{\ggm_{ik} m_k }{2\pi}
\le  \frac{1}{ c_1 m_{i,\eta}^+ } +\sum_{k=1}^3\frac{\ggm_{ik} M_k }{2\pi}
\le \frac{\# \{k:m_{i}^k >0\}}{ M_i } +\sum_{k=1}^3\frac{\ggm_{ik} M_k }{2\pi}.\]
Note that, crucially, the right hand side term is independent of $\eta$. Thus any 
$m_{i,\eta}^k$ satisfies
\begin{equation}
 \Lambda \ge  \frac{1}{r_{i,k}^\eta } \Lra  
r_{i,k}^\eta\ge \Lambda^{-1}\Lra 
m_{i,\eta}^k\ge c_2 \Lambda^{-2} \label{lower bound on the masses}
\end{equation}
again for some geometric constant $c_2$.

\medskip

{\em Step 2. Constructing a grid.} With the lower bound on the mass, we can now
focus our attention on bounding the lowest distance between two bubbles.
Denote by $C^k_\eta $ the bubbles of $v_{\eta}$. Note the main the difference between
$C^k_\eta $ and $\om^k_{i,\eta} $ is that the former does not distinguish between the different lobes
in the same bubbles, while the latter does. Denote by
\[ d_{k,h,\eta }:=d_\H (C^k_\eta,C^h_\eta) = \max \left\{
\sup_{z\in C^h_\eta}  \inf_{z'\in C^k_\eta} |z-z'|,
\sup_{z\in C^k_\eta}  \inf_{z'\in C^h_\eta} |z-z'|
\right\} \]
the Hausdorff distance between two bubbles $C^k_\eta $, $C^h_\eta $. Note the total
number of bubbles $C^k_\eta $ does not exceed the total number of lobes $\om^k_{i,\eta} $
across all the three constituent types. Denote by $K$ such quantity. 
Choose some $n\ge \lceil \sqrt{K} \rceil+1$. Clearly, as $K\le 3\# \{k:m_{i}^k >0\}$,
 such $n$
can also be taken independent of $\eta$. Partition $\mathbb{T}^2$ into $n^2$
squares $S_k$, each of which with side length $1/n$. 

\medskip

{\em Step 3. Constructing a competitor.} Denote by
\[ d_\eta^-:=\inf_{k\neq h} d_{k,h,\eta }\]
the smallest distance between two bubbles, and let  $(C^k_\eta)^- $, $(C^h_\eta)^- $ the corresponding
bubbles. We can always construct the following competitor $u_\eta$:
\begin{itemize}
	\item let $O_k$ (resp. $B_k$) be the barycenter of $S_k$ (resp. $C^k_\eta$ ). Replace 
	$C^k_\eta $ by $C^k_\eta +O_k-B_k $. This ensures that in each $S_k$, there is only one bubble
	$C^k_\eta $. 
\end{itemize}
We now find the energy $E_\eta(u_\eta)$: since the above construction is done through a translation,
all the perimeter and the self interaction terms
\[ \sum_{ k  } \int_{C_{\eta}^k \times C_{\eta}^k }   \left [ -\frac{1}{2\pi} \log |x-y| + R_{\mathbb{T}^2} (x - y)  \right ]  \d x\d y  \]
are unchanged. 
Since $v_\eta$ is a minimizer, we have $E_\eta(u_\eta)\ge E_\eta(v_\eta)$, thus we need
\begin{align}
 \sum_{ k\neq h  } \int_{ (C_{\eta}^k +O_k -B_k) \times (C_{\eta}^h +O_k-B_k) } &  \left [ -\frac{1}{2\pi} \log |x-y| + R_{\mathbb{T}^2} (x - y)  \right ]  \d x\d y\notag\\
 &\ge 
\sum_{ k\neq h  } \int_{C_{\eta}^k \times C_{\eta}^h  }   \left [ -\frac{1}{2\pi} \log |x-y| + R_{\mathbb{T}^2} (x - y)  \right ]  \d x\d y. \label{energy inequality requirement}
\end{align}
Since, by construction, $C_{\eta}^k +O_k-B_k$ is contained in $B(O_k, R\eta )$,
and the side length of $S_k$ is $1/n$, and $S_k$ contains no other bubble, we get 
\[\inf_{k\neq h} d_\H (C_{\eta}^h +O_k-B_k,C_{\eta}^k+O_k-B_k) \ge \frac{1}{n}-2R\eta
\ge \frac{1}{2n} \]
for sufficiently small $\eta$.
The left hand side term of \eqref{energy inequality requirement} can be estimated through
\begin{align}
\sum_{ k\neq h  }& \int_{C_{\eta}^k \times (C_{\eta}^h +O_k-B_k) }   \left [ -\frac{1}{2\pi} \log |x-y| + R_{\mathbb{T}^2} (x - y)  \right ]  \d x\d y\notag\\
&\le \sum_{ k\neq h  } \int_{C_{\eta}^k \times (C_{\eta}^h +O_k-B_k) }   \left[ -\frac{1}{2\pi} \log \frac{1}{2n}+ \sup R_{\mathbb{T}^2}  \right]  \d x\d y
=\left[ -\frac{1}{2\pi} \log \frac{1}{2n}+ \sup R_{\mathbb{T}^2}  \right] 
\sum_{ k\neq h  }|C_{\eta}^k| |C_{\eta}^h|.\label{energy grid 1}
\end{align}
Similarly, the right hand side  of \eqref{energy inequality requirement} can be estimated through
\begin{align}
\sum_{ k\neq h  } &\int_{C_{\eta}^k \times C_{\eta}^h  }   \left [ -\frac{1}{2\pi} \log |x-y| + R_{\mathbb{T}^2} (x - y)  \right ]  \d x\d y \notag\\
&\ge |(C^k_\eta)^-| |(C^k_\eta)^-|\left [ -\frac{1}{2\pi} \log ( d_\eta^- +2R\eta 
 ) +\inf R_{\mathbb{T}^2}  \right ] +  \sum_{ k\neq h  }|C_{\eta}^k| |C_{\eta}^h|
 \left[ -\frac{1}{2\pi} \log \sqrt{2}+ \inf R_{\mathbb{T}^2}  \right] .
 \label{energy grid 2}
\end{align}
Combining with \eqref{energy grid 1} gives the necessary condition
\begin{align*}
|(C^k_\eta)^-| |(C^k_\eta)^-|\bigg [ -\frac{1}{2\pi} \log ( d_\eta^- +2R\eta 
) +\inf R_{\mathbb{T}^2}  \bigg] &+  \sum_{ k\neq h  }|C_{\eta}^k| |C_{\eta}^h|
\left[ -\frac{1}{2\pi} \log \sqrt{2}+ \inf R_{\mathbb{T}^2}  \right] \\
&\le \left[ -\frac{1}{2\pi} \log \frac{1}{2n}+ \sup R_{\mathbb{T}^2}  \right] 
\sum_{ k\neq h  }|C_{\eta}^k| |C_{\eta}^h|.
\end{align*}
Finally, noting that the masses $|C_{\eta}^k|$ are bounded from above by $M_1+M_2+M_3$,
and from below by \eqref{lower bound on the masses}, both of which are clearly independent
of $\eta$, we conclude that $ d_\eta^-\ge \dt$ for some $\dt$ independent of $\eta$ too.
\end{proof}


\begin{thebibliography}{1}



\bibitem{afjm} E. Acerbi, N. Fusco, V. Julin and M. Morini. Nonlinear stability results for the modified Mullins-Sekerka and the surface diffusion flow. \textit{J. Differential Geom.}, 113(1):1-53, 2019.

\bibitem{afm} E. Acerbi, N. Fusco, and M. Morini. Minimality via second variation for a nonlocal isoperimetric problem. \textit{Comm. Math. Phys.}, 322(2): 515-557, 2013. 
\bibitem{ablw} S. Alama, L. Bronsard, X. Lu and C. Wang, Periodic minimizers of a ternary non-local isoperimetric problem. \textit{Indiana Univ. Math. J.}, 70(6):2557-2601,2021. 
\bibitem{nano} S. Alama, L. Bronsard, I. Topaloglu. Sharp Interface Limit of an Energy Modeling Nanoparticle-Polymer Blends. \textit{Interfaces and Free Boundaries}, 18(2): 263-290, 2016.
\bibitem{ABCT1} S. Alama, L. Bronsard, Rustum Choksi, I. Topaloglu. Droplet breakup in the liquid drop model with background potential. \textit{Commun. Contemp. Math.}, 21(3), 1850022, 2019.
\bibitem{ABCT2} S. Alama, L. Bronsard, Rustum Choksi, I. Topaloglu. Droplet phase in a nonlocal isoperimetric problem under confinement. \textit{Comm. Pure Appl. Anal.}, 19, 175-202, 2020.
\bibitem{otto} G. Alberti, R. Choksi and F. Otto. Uniform energy distribution for an isoperimetric problem with long-range interactions. \textit{J. Amer. Math. Soc.}, 22(2):569-605, 2009.
\bibitem{AFP} L. Ambrosio, N. Fusco and D. Pallara. Functions of Bounded Variation and Free Discontinuity Problems. \textit{The Clarendon Press, Oxford University Press}, New York, 2000.
\bibitem{bf} F.S. Bates and G.H. Fredrickson, Block copolymers - designer soft materials. \textit{ Phys. Today}, 52(2): 32-38, 1999.

\bibitem{BC} M. Bonacini and R. Cristoferi. Local and global minimality results for a nonlocal isoperimetric problem on $\mathbb{R}^N$. \textit{SIAM J. Math. Anal.}, 46(4): 2310-2349, 2014.
\bibitem{BK} M. Bonacini and H. Kn\"uepfer. Ground states of a ternary system including attractive and repulsive Coulomb-type interactions. \textit{Calc. Var. Partial Dif.}, 55:114, 2016.

\bibitem{cnt} R. Choksi, R. Neumayer and I. Topaloglu. Anisotropic liquid drop models, submitted.

\bibitem{bi1} R. Choksi and M.A. Peletier, Small Volume Fraction Limit of the Diblock Copolymer Problem: I. Sharp Interface Functional. \textit{SIAM J. Math. Anal.}, 42(3):1334-1370, 2010. 
\bibitem{onDerivation} R. Choksi and X. Ren. On the derivation of a density functional theory for microphase separation of diblock copolymers. \textit{J. Statist. Phys.}, 113:151-176, 2003.
\bibitem{blendCR} R. Choksi and X. Ren. Diblock copolymer-homopolymer blends: derivation of a density functional theory. \textit{Physica D}. 203(1-2):100-119, 2005.
\bibitem{cs} R. Choksi and P. Sternberg. On the first and second variations of a nonlocal isoperimetric problem. \textit{J. Reine Angew. Math.}, 611, 75-108, 2007.


\bibitem{cisp} M. Cicalese and E. Spadaro. Droplet minimizers of an isoperimetric problem with long-range interactions. \textit{Comm. Pure Appl. Math.}, 66(8):1298-1333, 2013.
\bibitem{cristoferi} R. Cristoferi. On periodic critical points and local minimizers of the Ohta-Kawasaki functional. \textit{Nonlinear Anal.}, 168, 81-109, 2018.
\bibitem{fall} M. Fall. Periodic patterns for a model invollving short-range and long-range interactions. \textit{Nonlinear Anal.}, 175:73-107, 2018. 

\bibitem{figalli} A. Figalli, N. Fusco, F. Maggi, V. Millot and M. Morini. Isoperimetry and stability properties of balls with respect to nonlocal energies. \textit{Comm. Math. Phys.}, 336(1):441-507, 2015.

\bibitem{FABHZ} J. Foisy, M. Alfaro, J. Brock, N. Hodges, and J. Zimba. The standard double soap bubble in $\mathbb{R}^2$ uniquely minimizes perimeter. \textit{Pacific J. Math.}, 159(1):47-59, 1993.

\bibitem{fl} R. Frank and E. Lieb. A compactness lemma and its application to the existence of minimizers for the liquid drop model.  \textit{SIAM J. Math. Anal.},47(6):4436-4450, 2015.


\bibitem{evolutionTer} K. Glasner. Evolution and competition of block copolymer nanoparticles. \textit{ SIAM J. Appl. Math.}, 79(1):28-54, 2019.
\bibitem{gc} K. Glasner and R. Choksi. Coarsening and self-organization in dilute diblock copolymer melts and mixtures. \textit{Phys. D}, 238:1241-1255, 2009.
\bibitem{GMSdensity} D. Goldman, C.B. Muratov and S. Serfaty. The $\Gamma$-limit of the two-dimensional Ohta-Kawasaki energy. I. droplet density. \textit{Arch. Rat. Mech. Anal.}, 210(2):581-613, 2013.
\bibitem{db1} J. Hass and R. Schlafly. Double bubbles minimize. \textit{Ann. Math.}, 151(2):459-515, 2000.
\bibitem{db2} M. Hutchings, R. Morgan, M. Ritor\'e, and A. Ros. Proof of the double bubble conjecture. \textit{Ann. Math.}, 155(2):459-489, 2002.
\bibitem{hnr} M. Helmers, B. Niethammer, and X. Ren. Evolution in off-critical diblock copolymer melts. \textit{Netw. Heterog. Media}, 3:615-632, 2008.

\bibitem{isenberg} C. Isenberg. The science of soap films and soap bubbles. Dover Publications, 1992. 


\bibitem{Julin} V. Julin, Isoperimetric problem with a Coulomb repulsive term. \textit{Indiana Univ. Math. J.}, 63(1): 77-89, 2014.


\bibitem{Julin2} V. Julin and G. Pisante. Minimality via second variation for microphase separation of diblock copolymer melts. \textit{J. Reine Angew. Math.}, 729, 81, 2017.
\bibitem{Julin3} V. Julin. Remark on a nonlocal isoperimetric problem. \textit{Nonlinear Anal.}, 154:174-188, 2017.


\bibitem{luotto} J. Lu and F. Otto. An isoperimetric problem with Coulomb repulsion and attraction to a background nucleus, preprint, arXiv:1508.07172, 2015.

\bibitem{knupfer1} H. Kn\"upfer and C.B. Muratov. On an isoperimetric problem with a competing non-local term. I. The planar case. \textit{ Comm. Pure Appl. Math.}, 66, 1129?1162, 2013.
\bibitem{knupfer2} H. Kn\"upfer and C.B. Muratov. On an isoperimetric problem with a competing non-local term. II. The general case. \textit{ Comm. Pure Appl. Math.}, 67, 1974-1994, 2014.
\bibitem{lions} P.L. Lions, The Concentration-Compactness Principle in the Calculus of Variations. The Limit Case, Part I. \textit{ Rev. Mat. Iberoamercana}, 1(1):145-201, 1984.
\bibitem{ms} M. Morini and P. Sternberg. Cascade of minimizers for a nonlocal isoperimetric problem in thin domains. \textit{SIAM J. Math. Anal.}, 46(3):2033-2051, 2014.
\bibitem{microphase} H. Nakazawa and T. Ohta. Microphase separation of ABC-type triblock copolymers. \textit{Macromolecules}, 26(20):5503-5511, 1993.
\bibitem{nishiura} Y. Nishiura and I. Ohnishi. Some mathematical aspects of the micro-phase separation in diblock copolymers. \textit{Phys. D}, 84, 31-39, 1995. 
\bibitem{maggi} F. Maggi. Sets of finite perimeter and geometric variational problems: An introduction to geometric measure theory. New York: Cambridge University Press. 2012.


\bibitem{Min} N. Min, T. Choi, S. Kim. Bicolored Janus microparticles created by phase separation in emulsion drops. \textit{Macromol. Chem. Phys.}, 218:1600265, 2017.
\bibitem{Mogi} Y. Mogi, {\it{et al.}}, Superlattice structures in morphologies of the ABC Triblock copolymers. \textit{Macromolecules}, 27(23): 6755-6760, 1994.


\bibitem{muratov} C.B. Muratov. Droplet phases in non-local Ginzburg-Landau models with Coulomb repulsion in two dimensions. \textit{Comm. Math. Phys.}, 299(1):45-87,2010.
\bibitem{oshita} Y. Oshita. Singular limit problem for some elliptic systems. \textit{SIAM J. Math. Anal.}, 38(6):1886-1911, 2007.
\bibitem{equilibrium} T. Ohta and K. Kawasaki. Equilibrium morphology of block copolymer melts. \textit{Macromolecules}, 19(10):2621-2632,1986.
\bibitem{stationary}  X. Ren and C. Wang. A stationary core-shell assembly in a ternary inhibitory system. \textit{Discrete Contin. Dyn. Syst.}, 37(2):983-1012, 2017. 
\bibitem{disc} X. Ren and C. Wang. Stationary disk assemblies in a ternary system with long range interaction. \textit{Commun. Contemp. Math.}, 1850046, 2018.
\bibitem{miniRW} X. Ren and J. Wei. On the multiplicity of solutions of two nonlocal variational problems. \textit{SIAM J. Math. Anal.}, 31(4):909-924, 2000.
\bibitem{lameRW} X. Ren and J.Wei. Triblock copolymer theory: Ordered ABC lamellar phase. \textit{J. Nonlinear Sci.}, 13(2):175-208, 2003.
\bibitem{many}  X. Ren and J.Wei.  Many dropet pattern in the cylindrical phase of diblock copolymer morphology. {\em  Rev. Math. Phys.}, 19(8):879-921, 2007.
\bibitem{doubleAs}  X.Ren and J.Wei. A double bubble assembly as a new phase of a ternary inhibitory system. {\em Arch. Rat. Mech. Anal.}, 215(3):967-1034, 2015.
\bibitem{double} X. Ren and J. Wei. A double bubble in a ternary system with inhibitory long range interaction. \textit{Arch. Rat. Mech. Anal.}, 46(4):2798-2852, 2014.
\bibitem{Schwarz} H. A. Schwarz. Beweis des Satze, dass die Kugel kleinere Oberfl\"ache besitzt, als jeder andere K\"orper gleichen Volumens. \textit{Nach. K\"oniglichen Ges. Wiss. G\"ottingen}, pages 1-13, 1884.
\bibitem{st}  P. Sternberg and I. Topaloglu. A note on the global minimizers of the nonlocal isoperimetric problem in two dimensions. \textit{Interfaces Free Bound.}, 13(1):155-169, 2011.
\bibitem{ihsan} I. Topaloglu. On a nonlocal isoperimetric problem on the two-sphere. \textit{Comm. Pure Appl. Anal.}, 12(1):597-620, 2013.




\bibitem{proof} W. Wichiramala. Proof of the planar triple bubble conjecture. \textit{J. Reine Angew. Math.}, 567:1-49, 2004.
\end{thebibliography}
 \end{document}